\documentclass[11pt,letterpaper]{amsart}
\usepackage[]{amsmath, amsthm, amsfonts, verbatim, amssymb}
\usepackage[backref=page]{hyperref} 
\hypersetup{colorlinks = true} 
\usepackage{rotating}
\usepackage{epstopdf}
\usepackage{pifont}
\usepackage{enumerate}
\usepackage{soul}
\usepackage{graphicx}
\usepackage{mathrsfs}\usepackage[all]{xy}
\DeclareMathAlphabet{\mathpzc}{OT1}{pzc}{m}{it}
\usepackage{tikz,tikz-cd,color}
\usepackage[makeroom]{cancel}
\usetikzlibrary{arrows}
\usepackage{lipsum} 
\renewcommand*{\backref}[1]{}
\renewcommand*{\backrefalt}[4]{%
	\ifcase #1 (Not cited.)%
	\or        (Cited on page~#2.)%
	\else      (Cited on pages~#2.)%
	\fi}

\newtheorem{theorem}{Theorem}[section]
\newtheorem*{theorem*}{Theorem}
\newtheorem{lemma}[theorem]{Lemma}
\newtheorem{proposition}[theorem]{Proposition}
\newtheorem{corollary}[theorem]{Corollary}
\newtheorem{conjecture}[theorem]{Conjecture}
\newtheorem{definition}[theorem]{Definition}

\newtheorem{problem}[theorem]{Problem}
\theoremstyle{remark}
\newtheorem{remark}[theorem]{Remark}
\newtheorem{example}[theorem]{Example}

\newtheorem*{claim}{Claim}

\newcommand{\bC}{\mathbb{C}}
\newcommand{\bF}{\mathbb{F}}

\newcommand{\bQ}{\mathbb{Q}}
\newcommand{\bP}{\mathbb{P}}
\newcommand{\bR}{\mathbb{R}}
\newcommand{\bZ}{\mathbb{Z}}

\newcommand{\cE}{\mathcal{E}}
\newcommand{\cF}{\mathcal{F}}
\newcommand{\cG}{\mathcal{G}}

\newcommand{\cI}{\mathcal{I}}

\newcommand{\cO}{\mathcal{O}}

\newcommand{\cQ}{\mathcal{Q}}

\newcommand{\Bl}{{\rm Bl}}

\newcommand{\rank}{{\rm rank}}

\newcommand{\Sing}{{\rm Sing}}

\def\<{\langle}
\def\>{\rangle}

\def\ra{\rightarrow}
\def\dra{\dashrightarrow}
\def\xra{\xrightarrow}

\author{Yen-An Chen}
\address{School of Mathematics, Korea Institute for Advanced Study, 85 Hoegi-ro, Dongdaemun-gu, Seoul 02455, Republic of Korea}
\email{yachen@kias.re.kr}

\author{Ching-Jui\ Lai}
\address[]{Department of Mathematics, National Cheng Kung University, Tainan 70101, Taiwan
}
\email{cjlai72@mail.ncku.edu.tw}

\begin{document}
\title[]{On Slope Unstable Fano Varieties}

\begin{abstract} For Fano varieties, significant progress has been made recently in the study of $K$-stability, while the understanding of the weaker but more algebraic concept of $(-K)$-slope stability remains intricate. For instance, a conjecture attributed to Iskovskikh states that the tangent bundle of a Picard rank one Fano manifold is slope stable. Peternell-Wi\'sniewski and Hwang proved this conjecture up to dimension five in 1998, but Kanemitsu later disproved it in 2021. To address this gap in understanding, we present a method that aims to characterize the geometry associated with the maximal destabilizing sheaf of the tangent sheaf of a Fano variety. This approach utilizes modern advancements in the foliated minimal model program. In dimension two, our approach leads to a complete classification of $(-K)$-slope unstable weak del Pezzo surfaces with canonical singularities. As by-products, we provide the first conceptual proof that \(\mathbb{P}^1 \times \mathbb{P}^1\) and \(\mathbb{F}_1\) are the only $(-K)$-slope unstable nonsingular del Pezzo surfaces, recovering a classical result of Fahlaoui in 1989. We also uncover a phenomenon that does not occur for Fano manifolds: there exists a del Pezzo surface with type A singularities admitting a weak K\"ahler-Einstein metric, yet whose tangent sheaf is slope unstable.
\end{abstract}

	\maketitle

\section{Introduction} 
We work over the field of complex numbers $\bC.$ 

A Fano (resp. weak Fano) manifold is a nonsingular projective variety $X$ whose anti-canonical divisor $-K_X$ is ample (resp. nef and big). Suppose $n=\dim X$. We say $T_X$ is $(-K_X)$-slope stable if for any subsheaf $\cE\subseteq T_X$, the slope inequality holds: $$\mu(\cE):=\frac{c_1(\cE)\cdot(-K_X)^{n-1}}{\rank(\cE)}<\mu(T_X)=\frac{(-K_X)^n}{n}.$$ An extremal contraction is a projective morphism $f:X\rightarrow Y$ with connected fibers such that $K_X$ is relatively anti-ample, i.e., $f$ is a Mori contraction of an extremal face (of $\dim\geq1$) of the cone of curves. We focus on the following conjecture in this work. 

\begin{conjecture}\cite[Conjecture 3.21]{P01:pos}\label{conj:P} The tangent bundle $T_X$ of a Fano manifold $X$ is $(-K_X)$-slope unstable only if the relative tangent bundle of a $K_X$-negative extremal contraction destabilizes it. 
\end{conjecture}

This conjecture is confirmed when $\dim X=2$ in \cite{Fah}, and $\dim X=3$ in \cite{St:F3}, both depending on the explicit classification of lower-dimensional Fano manifolds. Moreover, it is proved that $X=\bF_1$ and $\bF_0=\bP^1\times\bP^1$ are the only unstable nonsingular del Pezzo surfaces, where $T_{X/\bP^1}\subseteq T_X$ of the natural projections $X\rightarrow\bP^1$ is a destabilizing subsheaf. On the other hand, a folklore conjecture attributed to Iskovskikh, cf. \cite{A18:ICM}, asks whether the tangent bundle $T_X$ of a Fano manifold $X$ of Picard rank one is always $(-K_X)$-slope stable. This conjecture holds in $\dim X\leq5$ \cite{PW:stb2=1, H:StF} but fails in higher dimensions \cite{K:Fst} by a 14-dimensional horospherical Fano manifold $X$. Indeed, there is a non-trivial fibration $\widetilde{X}=\Bl_ZX\ra Y$ such that $T_{\widetilde{X}/Y}$ induces the maximal destabilizing sheaf (of rank two) of $T_X$. 

In the above discussion, in $\dim X=2$ the relative tangent $\cF=T_{X/\bP^1}$ is the \emph{maximal destabilizing sheaf} of the Harder-Narasimhan filtration of $T_X$, but this is not known in $\dim X\geq 3$. 
We consider the following generalization of Conjecture \ref{conj:P}. 
\begin{problem}\label{unst}\label{prob:main} For $X$ a Fano manifold with a $(-K_X)$-slope unstable tangent bundle, is the maximal destabilizing sheaf $\cF\subseteq T_X$ induced {\bf birationally} by the relative tangent of an extremal contraction? 
\end{problem}

A positive evidence is provided in \cite{SB92, H:StF}: the maximal destabilizing sheaf $\cF$ is an algebraically integrable foliation. We follow along this line by further utilizing the recent development of the foliation minimal model program (fMMP) \cite{CS:fMMP3, CS:fadj}, especially the analysis of foliated singularities \cite{Ch:lcfsing, CC:Tf} and a foliation sensible birational modification--the {\bf Property $(*)$ modification} (see \cite{ACSS:PosM} or Theorem \ref{prop:*}), to characterize the geometry of $(X,\cF)$. The canonical divisor $K_\cF$ of $\cF$ is defined by $\det(\cF)=\cO_X(-K_\cF)$ and it is not pseudo-effective (cf. Proposition \ref{prop:nPSEF}).  By running a $K_\cF$-MMP\footnote{\emph{Caution}: As $(X,\cF)$ might have singularities worse than $\cF$-dlt, the most general known setup in which one can run an fMMP, we instead run the classical log minimal model program. Nevertheless, this process only contracts $ \cF $-invariant curves (see Proposition \ref{prop:foliatedMMP}) and hence can be viewed as running an fMMP in a broader sense.} $X\dra Y$ on the Mori dream space $X$ and analyzing the foliated geometry associated with the Mori fiber space $f:Y\rightarrow Z$, we completely classify weak del Pezzo surfaces with unstable tangent bundles and provide the first conceptual proof of \cite{Fah}.    

We prepare some terminology before stating our results.
\begin{definition}\label{def:2blowup} 
    Suppose that $\pi: X\to B$ is a fibred surface, $p\in X$ a nonsingular point, and $F:=\pi^{-1}(\pi(p))$. The {\bf$2$-blowup} $\varphi: X''\rightarrow B$ with center $p$ is constructed as the following: let $\varphi': X'\to X$ be the nonsingular blowup at $p$ with the exceptional divisor $E$, $F'$ be the strict transform of $F$ on $X'$, and $p' := E\cap F'$, then consider $\varphi'': X''\to X'$ the nonsingular blowup at $p'$ and set $\varphi:=\varphi''\circ\varphi'$. 
\end{definition}

Recall \cite[Proposition 8.1.2]{Dol} that a nonsingular blowup $S'={\rm Bl}_pS\rightarrow S$ of a nonsingular weak del Pezzo surface $S$ remains weak del Pezzo if $K_S^2>1$ and $p$ is not on any $(-2)$-curve.

\begin{example}\label{eg:one_2blowup}
    Let $\varphi_1:X_{n,1}\rightarrow\bP^1$ be the $2$-blowup with center $p_1$ on the Hirzebruch surface $\pi:\bF_n\ra\bP^1$. For $n\in\{0,1,2\}$, $X_{n,1}$ is weak del Pezzo if $p_1$ is not on any $(-2)$-curve. If $\cF_{n,1}$ is the foliation induced by $T_{\varphi_1}$, then $\mu(\cF_{n,1})=3=\mu(T_{X_{n,1}})$ and it destabilizes $T_{X_{n,1}}$. Note that $\cF_{n,1}\neq T_{\varphi_1}$ since the fiber $\varphi^{-1}(\varphi(p_1))$ is non-reduced. 
\end{example}

\begin{example}\label{eg:two_2blowup}
For $n\in\{0,1,2\}$, let $\varphi_2:X_{n,2}\rightarrow\bP^1$ be the $2$-blowup with center $p_2\in X_{n,1}$, a surface from Example \ref{eg:one_2blowup}. If $p_2$ is neither on any $(-2)$-curve of $\bF_n$ nor the fiber containing $p_1$, then $X_{n,2}$ is weak del Pezzo and the foliation $\cF_{n,2}$ induced by $T_{\varphi_2}$ destabilizes $T_{X_{n,2}}$ with $\mu(\cF_{n,2})=2=\mu(T_{X_{n,2}})$. 
\end{example}

\begin{example}\label{eg:three_2blowup}
For $n\in\{0,1,2\}$, let $\varphi_3:X_{n,3}\rightarrow\bP^1$ be the $2$-blowup with center $p_3\in X_{n,2}$, a surface from Example \ref{eg:two_2blowup}. If $p_3$ is neither on any $(-2)$-curve of $\bF_n$ nor the fiber containing $p_1$ or $p_2$, then $X_{n,3}$ is weak del Pezzo and the foliation $\cF_{n,3}$ induced by $T_{\varphi_3}$ destabilizes $T_{X_{n,3}}$ with $\mu(\cF_{n,3})=1=\mu(T_{X_{n,3}})$. 
See Figure~\ref{fig:eg1.6} for a visualization. 
\end{example}

\begin{figure}
    \centering
    \begin{tikzpicture}
        \def\r{0.7}
        \draw (-5*\r,0) -- (5*\r,0);
        \draw (1*\r,-1*\r) -- (-1*\r,1*\r);
        \draw (-2*\r/3,\r/3) -- (-2*\r/3,8*\r/3);
        \draw (-1*\r,2*\r) -- (\r,4*\r);
        \draw[orange] (0*\r,-0.6*\r) circle (9pt) node {$-2$};
        \draw[orange] (0.1*\r,3.7*\r) circle (9pt) node {$-2$};
        \draw[orange] (-0.8*\r/3,1.5*\r) circle (9pt) node {$-1$};

        \draw (-2*\r,-1*\r) -- (-4*\r,1*\r);
        \draw (-11*\r/3,\r/3) -- (-11*\r/3,8*\r/3);
        \draw (-4*\r,2*\r) -- (-2*\r,4*\r);
        \draw[orange] (-3*\r,-0.6*\r) circle (9pt) node {$-2$};
        \draw[orange] (-2.9*\r,3.7*\r) circle (9pt) node {$-2$};
        \draw[orange] (-9.8*\r/3,1.5*\r) circle (9pt) node {$-1$};

        \draw (4*\r,-1*\r) -- (2*\r,1*\r);
        \draw (7*\r/3,\r/3) -- (7*\r/3,8*\r/3);
        \draw (2*\r,2*\r) -- (4*\r,4*\r);
        \draw[orange] (3*\r,-0.6*\r) circle (9pt) node {$-2$};
        \draw[orange] (3.1*\r,3.7*\r) circle (9pt) node {$-2$};
        \draw[orange] (8.2*\r/3,1.5*\r) circle (9pt) node {$-1$};
    \end{tikzpicture}
    \caption{A curve configuration of Example~\ref{eg:three_2blowup}. The horizontal line indicates the proper transform of the negative section on $\bF_n$, the top six lines indicate the exceptional curves, and the bottom three skew lines indicate the strict transform of the fiber of the canonical fibration on $\bF_n$.}\label{fig:eg1.6}
\end{figure}
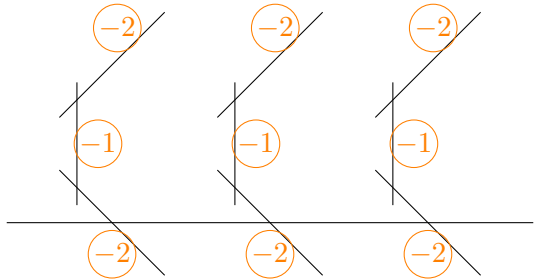

\begin{theorem}\label{thm:nonsingular-wdP}
Let $X$ be a nonsingular weak del Pezzo surface. If the tangent bundle $T_X$ is $(-K_X)$-slope unstable, then $X$ is either the Hirzebruch surface $\bF_n$ with $n\in\{0,1,2\}$, or $X\cong X_{n,1}, X_{n,2},$ or $X_{n,3}$ as described in Examples~\ref{eg:one_2blowup}, \ref{eg:two_2blowup}, and \ref{eg:three_2blowup}. In either case, $\cF$ is induced by the canonical fibration $X\rightarrow\bP^1$. 
\end{theorem}
By considering the minimal resolution, we immediately obtain the following. 

\begin{theorem}\label{thm:canwdp} Let $X$ be a weak del Pezzo surface with at worst canonical singularities and a $(-K_X)$-slope unstable tangent sheaf $T_X$. If $\cF$ is the maximal destabilizing sheaf, then $(X,\cF)$ is induced from its minimal resolution, which is one of the surfaces in Theorem \ref{thm:nonsingular-wdP}. 
\end{theorem}

As another consequence, we observe that, as in the nonsingular del Pezzo case, the tangent bundle is never $(-K_X)$-slope unstable.
\begin{corollary}\label{cor:semistable} For $X$ a weak del Pezzo surface with canonical singularities, the tangent sheaf $T_X$ is $(-K_X)$-slope semistable. 
\end{corollary}

By investigating possible configurations of $(-2)$-curves on the minimal resolution of a canonical weak del Pezzo surface, we have a characterization of possible singularities on canonical del Pezzo surfaces with an unstable tangent sheaf. 
\begin{corollary}\label{cor:candp} If $X$ is an $(-K_X)$-slope unstable del Pezzo surface with canonical singularities, then singularities of $X$ can only be one of the following types:
    \begin{enumerate}[$(1)$]
        \item $mA_1$ where $m=2$, $4$, or $6$,
        \item $A_1+A_2$,
        \item $2A_1+A_3$, and 
        \item $3A_1+D_4$.
    \end{enumerate}
\end{corollary}

This result reveals an interesting new phenomenon: For Fano manifolds, $K$-polystability of ${\displaystyle (X,-K_{X})}$ implies $(-K_X)$-slope polystabililty of $T_X$, by the Kobayashi-Hitchin correspondence. For singular Fano varieties, Odaka et al. \cite{OSS:Compact} have shown that a del Pezzo surface with canonical singularities of type $4A_1$ or $6A_1$ admits a weak KE metric. However, Corollary \ref{cor:candp} implies that the implication of $K$-polystable to $(-K_X)$-polystable can fail in the theory of weak KE metric on singular Fano varieties.


\begin{corollary}\label{cor:wKE} There exists a del Pezzo surface $X$, with $\Sing(X)=4A_1$ or $6A_1$, admitting a weak KE metric but whose tangent sheaf is slope unstable. (See Figure~\ref{fig:6A_1} for a visualization of the case when $\Sing(X)=6A_1$.)
\end{corollary}

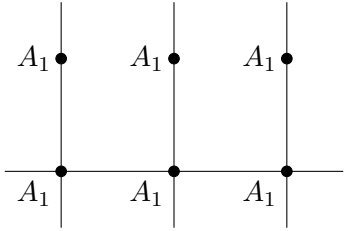
\begin{figure}
    \centering
    \begin{tikzpicture}
        \def\r{1.5}
        \draw (-0.5*\r,0) -- (2.5*\r,0);
        \draw (0,-0.5*\r) -- (0,1.5*\r);
        \draw (\r,-0.5*\r) -- (\r,1.5*\r);
        \draw (2*\r,-0.5*\r) -- (2*\r,1.5*\r);
        \filldraw[black] (0,0) circle (2pt) node[anchor=north east] {$A_1$};
        \filldraw[black] (0,\r) circle (2pt) node[anchor=east] {$A_1$};
        \filldraw[black] (\r,0) circle (2pt) node[anchor=north east] {$A_1$};
        \filldraw[black] (\r,\r) circle (2pt) node[anchor=east] {$A_1$};
        \filldraw[black] (2*\r,0) circle (2pt) node[anchor=north east] {$A_1$};
        \filldraw[black] (2*\r,\r) circle (2pt) node[anchor=east] {$A_1$};
    \end{tikzpicture}
    \caption{A visualization of a singular del Pezzo surface $X$ with $\operatorname{Sing}(X)=6A_1$ admitting a weak KE metric but whose tangent sheaf is slope unstable. This is obtained by contracting all $(-2)$-curves of Example~\ref{eg:two_2blowup} when $n=1$.}\label{fig:6A_1}
\end{figure}

\subsubsection*{Notations} For ease of readability, we use $p_+\in Y_\pm$, $W_\pm$, $B_\pm$, $F_\pm$, and $\cF_\pm$ to indicate constructions associated with $\pm K_\cF,$ even though $-K_\cF$ does not appear in this work.

\subsection{Sketch of the proof of Theorem \ref{thm:nonsingular-wdP}} As indicated before, let $\varphi:X\dra Y_+$ be a $K_\cF$-MMP on the Mori dream space $X$ that terminates at a Mori fiber space $f:Y_+\ra Z_+$. If $\dim Z_+=1$, then $Z_+=\bP^1$ and it is not hard to explicitly classify $X$ when $X$ is nonsingular by considering the minimal resolution of $Y_+$. The central technical part of the whole proof is to rule out the case $\dim Z_+=0,$ for which $Y_+$ is an unstable del Pezzo surface of Picard rank one with canonical singularities. 

If $\rho(Y_+)=1$, it is known that $\cF_+$ has a foliated non-klt center \cite{AD:Ff}. Our first observation is that in this setting, $(Y_+,\cF_+)$ indeed is toric with a unique dicritical singularity supported at a smooth point: To understand the geometry of $(Y_+,\cF_+)$, we consider a Property $(*)$ modification $W_+\rightarrow\bP^1$ \cite{ACSS:PosM} to analyze the singularities of $\cF_+$. We then show that $\cF_+$ has a unique dicritical singular point $p_+\in Y_+$, $p_+$ is nonsingular, and $\cF_+$ has foliated log canonical singularities by utilizing a delicate analysis of foliated surface singularities \cite{Ch:lcfsing} and a foliated contraction theorem for terminal foliations \cite{CS:fMMPr1}. For proving that $(Y_+,\cF_+)$ is a toric foliation, though an elementary approach exists, we exhibit an argument using a characterization of toric varieties via Shokurov's complexity \cite{BMSZ:Toric}. Finally, a case-by-case study of possible toric del Pezzo surfaces disproves the $\dim Z_+=0$ scenario. 

\subsection{Further discussion} As discussed in this work, the theory of foliated birational geometry offers a framework for understanding the slope instability of Fano manifolds. Our approach, which utilizes a foliation-sensitive birational modification, suggests a potential way to tackle Problem \ref{prob:main}. Though a generalization of Theorem \ref{thm:nonsingular-wdP} in dimension three is already intricate due to the appearance of singularities and the complexity of geometry. 

On the other hand, we propose the following seemingly more approachable problem, provided the technicalities of foliated geometry can be extended to higher dimensions. 

\begin{problem} For a Fano manifold with a $(-K)$-slope unstable tangent bundle, its maximal destabilizing sheaf of the tangent bundle has rank at least two. 
\end{problem}

Note that this problem has a positive answer when the Picard number is one (cf. \cite[Proof of Proposition 1]{H:StF}), and is compatible with the work of \cite{PW:stb2=1, H:StF}, the examples in \cite{K:Fst}, and Corollary \ref{cor:semistable}.

\subsubsection*{Structure of the paper} This paper is organized as the following: In Section \ref{sec:pre}, we review foliated singularities and toric foliations. In Section \ref{sec:fMMP}, we prove some general properties for the foliation induced by the maximal destabilizing sheaf and show a foliated minimal model program ends with a Mori fiber space $Y_+\to Z_+$. In Section \ref{sec:classification}, we classify nonsingular del Pezzo surfaces with unstable tangent bundles when $\dim Z_+=1$. In Section \ref{sec:nonexist}, we rule out the case $\dim Z=0$ and hence finish the proof of Theorem \ref{thm:nonsingular-wdP}.

\subsection*{Acknowledgements}  
We want to express our gratitude to Paolo Cascini, Hsin-Ku Chen, Jungkai Alfred Chen, Calum Spicer, Sebasti\'an Velazquez, and Kento Fujita for their valuable discussions. The first author extends their appreciation to National Taiwan University for providing an excellent research environment where part of this work was conducted. The second author thanks Paolo Cascini and Imperial College London for their hospitality during a research visit that contributed to this work. This project is supported by grants MOST-111-2115-M-006-002-MY2 and NSTC-113-2115-M-006-015-MY3 from the National Science and Technology Council in Taiwan.

\section{Preliminaries}\label{sec:pre}
\subsection{Foliations} We recall the notions of foliation, foliated singularities, a classification of foliated log canonical surface singularities, and toric foliation.  

\begin{definition}
    A \emph{foliation} $\cF$ on a normal variety $X$ is a subsheaf of the tangent sheaf $T_X$ such that 
    \begin{itemize}
        \item $\cF$ is saturated, that is, $T_X/\cF$ is torsion-free, and
        \item $\cF$ is closed under the Lie bracket. 
    \end{itemize}
\end{definition}
If $\cF=0$, then we call $\cF$ is the foliation by points. 

Let $f: Y\dashrightarrow X$ be a dominant rational map and $\cF$ be a foliation. 
We denote the \emph{pullback} of $\cF$ via $f$ by $f^{-1}\cF$ as constructed in \cite[Section 3.2]{D:C1I}. If $f$ is birational and $\cG$ is a foliation on $Y$, then we define the \emph{pushforward} of $\cG$ via $f$ as $(f^{-1})^{-1}\cG$, denoted by $f_*\cG$. 

A foliation $\cF$ on $X$ is called \emph{algebraically integrable} if $\cF$ is the pullback foliation $\varphi^{-1}\cF_Z$ of a dominant rational map $\varphi: X\dashrightarrow Z$ and $\cF_Z$ is the foliation by points on $Z$. We also say $\cF$ is induced by $\varphi$. 

If $\cF$ is a foliation of rank $r$, then we define the \emph{canonical divisor} of $\cF$ to be a Weil divisor $K_\cF$ such that $\det\cF=\cO_X(-K_\cF)$. 
This induces a morphism $\phi:(\wedge^r\Omega_X\otimes\cO_X(-K_\cF))^{\vee\vee}\to\cO_X$ and we define the \emph{singular locus} of $\cF$, denoted by $\operatorname{Sing}(\cF)$, as the co-support of the image of $\phi$. 

A subvariety $S\subseteq X$ is called \emph{$\cF$-invariant} if $\partial\cI_{S\cap U}\subseteq\cI_{S\cap U}$ for any open subset $U\subseteq X$ and any section $\partial\in H^0(U,\cF)$, where $\cI_{S\cap U}$ is the ideal sheaf of $S\cap U$ in $U$. 

\begin{definition}
    Let $\cF$ be a foliation on a normal variety $X$. 
    Let $f:Y\to X$ be a birational morphism and $E$ be a divisor on $Y$. 
    We define the discrepancy of $E$ with respect to $\cF$ to be $a(E,\cF)=\operatorname{mult}_E(K_{f^{-1}\cF}-f^*K_\cF)$. 
    We say $\cF$ is terminal (resp. canonical; resp. log canonical) if $a(E,\cF)>0$ (resp. $\geq 0$; resp. $\geq -\epsilon(E)$) for all prime divisors $E$ over $X$, where $\epsilon(E)=0$ if $E$ is $\cF$-invariant and $1$ otherwise. 
\end{definition}

\subsection{Log canonical foliated surface singularities}
In this subsection, we fix some notations from \cite{B:BGf} and recall the classification of log canonical foliated surfaces singularities in \cite{Ch:lcfsing}, which is an extension of Alekseev's classical arithmetic classification of log canonical surface singularities. 

Let $\cF$ be a rank one foliation on a normal surface $X$ and $p\in\operatorname{Sing}(\cF)\setminus\operatorname{Sing}(X)$. 
Since $p\notin\operatorname{Sing}(X)$, we let $v$ be a vector field around $p$ generating $\cF$. 
By definition  $p\in\operatorname{Sing}(\cF)$ if and only if $v(p)=0$ and we consider eigenvalues $\lambda_1$, $\lambda_2$ of the linear part $(\operatorname{D}v)\vert_p$, which do not depend on the choice of $v$. 
If one of the eigenvalues is non-zero, say $\lambda_2$, then define the \emph{eigenvalue of $\cF$ at $p$} to be $\lambda:=\frac{\lambda_1}{\lambda_2}$. 
For $\lambda\neq 0$, it is well-defined up to reciprocal $\lambda\sim\frac{1}{\lambda}$. 

A singular point $p\in X$ is called \emph{non-degenerate} if $\lambda\neq 0$ and  \emph{reduced} if $\lambda\notin\bQ_+$, respectively. 

\begin{definition}
    Let $\cF$ be a foliation on a normal surface and $C$ be an $\cF$-invariant curve. 
    Let $p\notin\operatorname{Sing}(X)$, $\omega$ be a $1$-form defining $\cF$ around $p$, and $f$ be the local defining function of $C$ around $p$.   
    Then we can write $g\omega = h\textnormal{d}f+f\eta$ for some holomorphic functions $g$, $h$ and holomorphic $1$-form $\eta$ such that $h$ and $f$ are relatively prime. 
    
    We define the index $\operatorname{Z}(\cF,C,p)$ to be the vanishing order of $\frac{h}{g}\vert_C$ at $p$. 
    This definition is independent of $f$, $g$, $h$, $\omega$, and $\eta$, cf. \cite[Chapter 2]{B:BGf}.
\end{definition}

Note that $\operatorname{Z}(\cF,C,p) = 0$ if $p\notin\operatorname{Sing}(\cF)\cup\operatorname{Sing}(X)$. 
Thus, if $C$ is compact and contained in the nonsingular locus of $X$, then we can define 
\[\operatorname{Z}(\cF,C):=\sum_{p\in C}\operatorname{Z}(\cF,C,p).\]

\begin{definition}
    Let $\cF$ be a foliation on a normal surface $X$. 
    We call a birational morphism $\pi: Y \to X$ is a \emph{resolution} of the foliation $\cF$ if $\pi^{-1}\cF$ has only reduced singularities. 

    A resolution $\pi: Y\to X$ of the foliation $\cF$ is \emph{minimal} if any resolution $\phi:Z\to X$ of $\cF$ factors through $\pi$. 
\end{definition}
\begin{remark}
    Let $\cF$ be a foliation on a normal surface. 
    A resolution of $\cF$ always exists by Seidenberg's theorem. 
    Also, there is a unique minimal resolution of $\cF$ up to isomorphism by \cite[Proposition 2.20]{Ch:lcfsing}. 
\end{remark}

\begin{definition}
    A compact curve $C=\bigcup_{i=1}^s C_i$ is called a \emph{string} if all $C_i$'s are nonsingular rational irreducible curves and $C_i\cdot C_j = 1$ if $|i-j|=1$ and $0$ if $|i-j|\geq 2$. 
    If moreover $C_i^2\leq -2$ for all $i$, then $C$ is called a \emph{Hirzebruch-Jung string}.
\end{definition}
\begin{definition}
    Let $\cF$ be a foliation on a nonsingular surface $X$. 
    \begin{enumerate}[$(1)$]
        \item An irreducible curve $C$ is called a $(-1)$-$\cF$-curve (resp. $(-2)$-$\cF$-curve) if 
        \begin{enumerate}[$(a)$]
            \item $C$ is a nonsingular rational $\cF$-invariant curve and 
            \item $\operatorname{Z}(\cF,C)=1$ (resp. $\operatorname{Z}(\cF,C)=2$). 
        \end{enumerate}
        \item We say a Hirzebruch-Jung string $C=\bigcup_{i=1}^sC_i$ is an $\cF$-chain if 
        \begin{enumerate}[$(a)$]
            \item each $C_i$ is $\cF$-invariant, 
            \item $\operatorname{Sing}(\cF)\cap C$ are all reduced and non-degenerate, and
            \item $\operatorname{Z}(\cF,C_1)=1$ and $\operatorname{Z}(\cF,C_i)=2$ for $i\geq 2$. 
        \end{enumerate}
        \item If an irreducible $\cF$-invariant curve $E$ meeting an $\cF$-chain $C$ but not contained in $C$, then $E$ is called a \emph{tail} of $C$. 
        Note that such $E$ is unique if it exists. 
        \item $E$ is called a bad tail if 
        \begin{enumerate}[$(a)$]
            \item $E$ is a nonsingular rational irreducible $\cF$-invariant curve with $\operatorname{Z}(\cF,E)=3$ and $E^2\leq -2$, and 
            \item $E$ intersects two $(-1)$-$\cF$-curves whose self-intersections are $-2$. 
        \end{enumerate}

        \item An elliptic Gorenstein leaf is either an $\cF$-invariant rational curve with only one node or a cycle of $(-2)$-$\cF$-curves. 
    \end{enumerate}
\end{definition}

We restate the classification of log canonical foliated surface singularities \cite{Ch:lcfsing} into two parts, canonical and non-canonical, for ease of later use. 
\begin{theorem}\label{can_sing}
    Let $\cF$ be a foliation with canonical singularities on a normal surface $X$ and $\pi: Y\to X$ be the minimal resolution of the $\cF$. Let $\cG:=\pi^{-1}\cF$ be the pullback foliation. 
    Then, the connected components of exceptional divisors of $\pi$ belong to one of the following types:
    \begin{enumerate}[$(1)$]
        \item A $\cG$-chain $\bigcup_{i=1}^sC_i$ with $a(C_s,\cF)=\frac{1}{r}$ where $r$ is the index of $K_\cF$ at the point $\pi(\bigcup_{i=1}^sC_i)$. 
        \item A chain of three invariant curves $E_1\cup E_2\cup E_3$ with $a(E_1,\cF)=a(E_3,\cF)=\frac{1}{2}$ and $a(E_2,\cF)=0$, where $E_1$ and $E_3$ are $(-1)$-$\cG$-curves with self-intersection $-2$ and $E_2$ is a bad tail. 
        \item A chain of $(-2)$-$\cG$-curves $\bigcup_{i=1}^sC_i$ with $a(C_i,\cF)=0$ for all $i$.
        \item A curve whose dual graph is of $D$ type. More precisely, two $(-1)$-$\cG$-curves $C_1$, $C_2$ whose self-intersections are $-2$ joined by a bad tail $C_3$, which itself connects to a chain of $(-2)$-$\cG$-curves $\bigcup_{i=4}^sC_i$. Here $a(C_1,\cF)=a(C_2,\cF)=\frac{1}{2}$ and $a(C_i,\cF)=0$ for all $i\geq 3$. 
        \item An elliptic Gorenstein leaf $\bigcup_{i=1}^s C_i$ with $a(C_i,\cF)=0$ for all $i$.
    \end{enumerate}
\end{theorem}
\begin{proof}
    We sketch the proof for the discrepancies. 
    Let $\pi^*K_\cF=K_\cG+\sum_ia_iC_i$ where $C_i$ are all $\pi$-exceptional prime divisors. 
    Note that the self-intersection matrix $(C_i\cdot C_j)$ is negative definite and we consider $(K_\cG+\sum_ia_iC_i)\cdot C_j = 0$ for all $j$. 
    For (1), note that $(\sum_ia_iC_i)\cdot C_1 = 1$ and $(\sum_ia_iC_i)\cdot C_j = 0$ for $j\geq 2$. 
    Thus 
    \[a(C_s,\cF)=-a_s=\frac{1}{\det(C_i\cdot C_j)}=\frac{1}{r}\]
    where $r$ is the index of $K_\cF$ at the point $\pi(\bigcup_{i=1}^sC_i)$. 
    For (2)-(5), it is a direct check that $a(G_i,\cF)$ have the prescribed values as $a(G_i,\cF)=-a_i$ for all $i$. 
\end{proof}

\begin{theorem}\label{lc_sing}
    Let $\cF$ be a foliation on a normal surface $X$ and $p\in X$ a log canonical non-canonical singularities for $\cF$. 
    Let $\pi: Y \to X$ be the minimal resolution of $\cF$ over $p$ and $\cG=\pi^{-1}\cF$ be the pullback foliation. 
    Then the exceptional divisor $\bigcup_{i\geq0}E_i$ of $\pi$ over $p$ has the following properties:
    \begin{enumerate}[$(1)$]
        \item There is exactly one irreducible component $E_0$ which is not $\cG$-invariant and $(K_\cG+E_0)\cdot E_0=0$, 
        \item $\bigcup_{i\geq1}E_i$ is a disjoint union of $\cG$-chains, and 
        \item for $i\geq 1$, $E_0\cdot E_i=1$ if $E_i$ is a $(-1)$-$\cG$-curve and $E_0\cdot E_i=0$ if $E_i$ is a $(-2)$-$\cG$-curve. 
    \end{enumerate}
\end{theorem}

\begin{lemma}[Precise adjunction]\label{lem:adjunction}
    Let $\cF$ be a foliation of rank one with canonical singularities on a normal projective surface $X$. 
    Let $C$ be an $\cF$-invariant curve with $\nu:C^\nu\rightarrow C$ the normalization map. 
    Then there is an effective divisor ${\rm Diff}_{C}(\cF)$ on $C^\nu$ such that $K_\cF\vert_{C^\nu} = K_{C^\nu}+{\rm Diff}_{C}(\cF)$, where $\operatorname{mult}_P{\rm Diff}_{C}(\cF)$ has the following form: 
    \begin{enumerate}[$(1)$]
        \item $\frac{r-1}{r}$ if $\nu(P)$ is a terminal foliation singularity and a cyclic quotient singularity of index $r$; 
        \item $1$ if $\nu(P)\in\Sing(X)$ and is a canonical non-terminal foliation singularity;
        \item $\operatorname{Z}(\cF,C,P)$ if $\nu(P)\in X$ is a nonsingular point. 
    \end{enumerate}
\end{lemma}
\begin{proof}
    Let $\pi: X'\to X$ be the minimal resolution of $\cF$ with $\cF'=\pi^{-1}\cF$ and $C'=\pi_*^{-1}C$. 
    By Theorem~\ref{can_sing}, $\cF'$ has at worst reduced singularities and $C'$ is nonsingular. 
    Then by \cite[Proposition 2.3]{B:BGf} and its proof, we have $K_{\cF'}\vert_{C'}\sim K_{C'}+\sum_{P'\in C'}\operatorname{Z}(\cF',C',P')P'$. 
    
    Write $K_{\cF'} = \pi^*K_\cF+E$, where $E$ is effective as $\cF$ has at worst canonical singularities. By Theorem~\ref{can_sing}, it is a direct computation that the support of $E$ is the union of $\cF'$-chains and $E\vert_{C'} = \sum_{P'\in E\cap C'} \frac{1}{r_{P'}}P'$ where $r_{P'}$ is the index of $K_\cF$ at $\nu(P')$. 
    
    Let $\pi_C: C'\to C^\nu$ be morphism induced from $\pi$.
    By \cite[Remark 3.10]{CS:fadj}, we have ${\rm Diff}_{C}(\cF)=(\pi_C)_*\operatorname{Diff}_{C'}(\cF',-E)$ and hence $\operatorname{mult}_P\Delta$ satisfies $(1)$-$(3)$. 
\end{proof}

\begin{lemma}[{Adjunction for non-invariant divisors}]\label{lem:adj_non-inv}
    Let $\cF$ be a foliation of rank one on a normal $\bQ$-factorial projective surface $X$. 
    Let $C$ be a non-$\cF$-invariant curve with $\nu: C^\nu\to C$ the normalization map. 
    Then $(K_\cF+C)\vert_{C^\nu} = \operatorname{Diff}_C(\cF)$ is an effective divisor on $C^\nu$. 
\end{lemma}
\begin{proof}
    This is a special case of \cite[Proposition-Definition 3.6]{CS:fadj}. 
\end{proof}

\subsection{Toric varieties and toric foliations}
All toric varieties are assumed to be normal. For notations of toric varieties, we follow \cite{CLS}. 

Let $N=\bZ^n$ be a lattice of rank $n$ and $M:=\operatorname{Hom}(N,\bZ)$ be the dual lattice. We write $N_\bR=N\otimes_\bZ\bR$, $N_\bC=N\otimes_\bZ\bC$, and $M_\bR = M\otimes_\bZ\bR$. A \emph{fan} $\Sigma$ in $N_\bR$ is a finite collection of rational strongly convex polyhedral cones. 
For any non-negative integer $k$, we denote the set of $k$-dimensional cones in $\Sigma$ as $\Sigma(k)$ and the set of $k$-dimensional faces of $\sigma\in\Sigma$ as $\sigma(k)$. The toric variety $X_{\Sigma}$ of the fan $\Sigma$ in $N_\bR$ is constructed by gluing all affine variety $U_{\sigma,N}= \operatorname{Spec}(\bC[\chi^m\mid m\in\sigma^\vee\cap M])$ together, where $\sigma^\vee\subseteq M_\bR$ is the dual cone of $\sigma$. For each $\sigma\in\Sigma$, $O_\sigma$ denotes the $T$-orbit of the
distinguished point $x_\sigma$, and $V_\sigma$ denotes the closure of $O_\sigma$ in $X_\Sigma$. 
If $\rho\in\Sigma(1)$ is a ray, then $V_\rho$ is a divisor and will also be denoted as $D_\rho$.

For the theory of toric foliations, we follow \cite{CC:Tf}. A foliation $\cF$ on a toric variety $X_\Sigma$ is called \emph{toric} if it is equivariant under the torus action. 
\begin{proposition}[{\cite[Proposition 3.1]{CC:Tf}}]\label{prop:toric-ss}
    Let $X_\Sigma$ be a toric variety of a fan $\Sigma$ in $N_\bR$. 
    Then, there is a one-to-one correspondence between the set of toric foliations on $X_\Sigma$ and the set of complex vector subspaces $V$ of $N_\bC$. 
\end{proposition}
We will use $\cF_{V,\Sigma,N}$ to denote the toric foliation associated with the complex vector subspace $V$ of $N_\bC$. 
If $\Sigma$ or $N$ are clear in the context, we will omit the subscriptions $\Sigma$ or $N$. 

\begin{proposition}[{\cite[Proposition 3.7]{CC:Tf}}]\label{prop:toric_fol_can_div}
    Let $\Sigma$ be a fan in $N_\bR$, $V\subseteq N_\bC$ be a complex vector subspace, and $\cF:=\cF_V$ be the corresponding toric foliation on $X_\Sigma$. 
    Then $K_\cF=-\sum_{\rho\in\Sigma(1),\,\rho\subseteq V}D_\rho$.
\end{proposition}

\begin{proposition}[{\cite[Proposition 3.9]{CC:Tf}}]\label{prop:tor_fol_sing}
    Let $\cF_V$ be a foliation on a $\bQ$-factorial toric variety $X_\Sigma$ defined by a fan $\Sigma$ in $N_\bR$, where $V\subseteq N_\bC$ is a complex vector subspace. Then for any $\tau\in\Sigma$, $V_\tau\nsubseteq\operatorname{Sing}(\cF_V)$ if and only if $V\cap\bC\tau=\operatorname{Span}(S)$ for some $S\subseteq\tau(1)$ with the convention $\operatorname{Span}(\emptyset)=0$. 
\end{proposition}

\section{A foliated MMP on an unstable weak Fano variety}\label{sec:fMMP}
\subsection{Maximal destabilizing sheaf} Let $X$ be a weak $\bQ$-Fano variety with at worst canonical singularities, that is $-K_X$ is $\bQ$-Cartier and nef and big. Consider the slope stability of $T_X$ with respect to $-K_X$. This is equivalent to the $\alpha$-slope stability with respect to the movable curve class $\alpha$ associated to $|m(-K_X)|$, and by \cite{CP:GS} there exists a unique Harder-Narashimhan filtration 
$$0=\cF_0\subsetneq\cF_1\subsetneq\cdots\subsetneq\cF_r=T_{X}$$
such that 
$\cQ_i:=\cF_i/\cF_{i-1}$ is $(-K_X)$-slope semistable for $i=1,\dots, r$ and with 
$$\mu_{\max}(T_X):=\mu(\cQ_1)>\cdots>\mu(\cQ_r):=\mu_{\min}(T_X).$$
Recall that each $\cF_i$ is defined so that $\cQ_i\subseteq T_X/\cF_{i-1}$ is the maximal destabilizing sheaf. In particular, 
$\mu_{\max}(T_X)=\mu(\cF_1)\geq\mu(T_X)>0.$

\begin{proposition}\label{prop:nPSEF} With the above setup, the maximal destabilizing sheaf $\cF$ is an algebraically integrable foliation with rationally connected leaves, and the canonical divisor $K_\cF$ is not pseudo-effective.
\end{proposition}
\begin{proof} Note that by \cite[Proposition 22]{O:gnef}, a generalization of \cite[Theorem 1.4]{CP:F}, each $\cQ_i$ with $\mu(\cQ_i)>0$ is an algebraically integrable foliation with rationally connected leaves. 
    
Suppose $K_\cF$ is pseudo-effective. As $-K_X$ is nef, this implies that 
\begin{align*}
        0 &\geq \frac{(-K_\cF)\cdot(-K_X)^{n-1}}{\rank(\cF)}= \mu(\cF)\geq \mu(T_{X})= \frac{(-K_X)^n}{n} >0, 
    \end{align*}
which is impossible. 
\end{proof}

\begin{remark} Let $\varphi:X\rightarrow \overline{X}$ be defined by $|-mK_X|$ for $m\gg0$ and divisible, i.e., $\overline{X}$ is the anticanonical model of $X$, then $\varphi^*K_{\overline{X}}=K_X$ and $\overline{X}$ has at worst canonical singularities. Let $\overline{\cF}_i$'s and $\overline{\cQ}_i$'s, $i=0,\dots, \overline{r}$, be the Harder-Narashimhan filtration and semistable pieces of $T_{\overline{X}}$. It is easy to see that $\overline{r}=r$ and the reflexive sheaves $\overline{\cF}_i$'s and $\overline{\cQ}_i$'s are induced respectively from $\cF_i$'s and $\cQ_i$'s and vice verse.   
\end{remark}

\subsection{A foliated MMP} Recall that log Fano varieties are Mori dream spaces by \cite{BCHM}. The following lemma shows that by starting with a weak $\bQ$-Fano variety $X$, any intermediate step of a $D$-MMP of $X$ is of Fano type, with, at worst, canonical singularities. 

\begin{lemma}\label{D_MMP_canonical} 
Let $X$ be a weak $\bQ$-Fano variety and the birational map $X\dashrightarrow Y$ be a $D$-MMP. Then, $Y$ is $\bQ$-factorial, log Fano type, with at worst canonical singularities. 
\end{lemma}
\begin{proof} Since $(X,B)$ is log Fano for some $B\geq0$, we can run a $D$-MMP of any given divisor $D$ as a $(K_X+B+\Delta)$-MMP by taking a general element $mN\Delta\in|m(D-N(K_X+B))|$ for some $N\gg0$ and a divisible $m\gg0$. 

By the base point freeness theorem, there is an effective divisor $\Gamma\sim_\bQ-K_X$ such that $(X,\Gamma)$ is canonical. 
It follows from the negativity lemma that $p^*(K_X+\Gamma)\sim_\bQ q^*(K_Y+\Gamma_Y)$ on a common resolution $X\xleftarrow{p} W\xra{q}Y$, where $\Gamma_Y$ is the strict transform of $\Gamma$ on $Y$. 
Hence, $(Y,\Gamma_Y)$ is canonical. Since $Y$ is $\bQ$-factorial, $Y$ is also canonical. The rest is standard.
\end{proof}

Consider the maximal destabilizing sheaf $\cF$ of the tangent sheaf of a weak $\bQ$-Fano variety $X$ and take $D=K_\cF$. By Proposition \ref{prop:nPSEF}, $K_\cF$ is not pseudo-effective and there is the following diagram: 
\begin{center}
\begin{tikzcd}  X\arrow[dashed]{r}{\varphi_+} & Y_+\arrow{d}{f_+}\\
 & Z_+,
\end{tikzcd}
\end{center}
where $\varphi_+:X\dashrightarrow Y_+$ is a $K_\cF$-MMP and $f_+:Y_+\rightarrow Z_+$ is the associated Mori fiber space.

\subsection{Weak del Pezzo surfaces}\label{subsec:mmp-wdP} When $X$ is a $(-K_X)$-slope unstable weak del Pezzo surface with canonical singularities, the $K_\cF$-MMP $\varphi_+:X\dra Y_+$ is a morphism and we observe that it is indeed a foliated MMP, i.e., it contracts only $\cF$-invariant curves. 

\begin{proposition}\label{prop:foliatedMMP}
    Let $X$ be a normal $\bQ$-factorial surface and $\cF$ be a foliation on $X$. If $\varphi:X\to Y$ is a morphism contracting a $K_\cF$-negative extremal curve $C$, then $C$ is $\cF$-invariant. 
\end{proposition}
\begin{proof} As $C$ can be contracted, we have $C^2\leq0$. Suppose for the sake of contradiction that $C$ is not $\cF$-invariant. Then we have 
    \[0>(K_\cF+C)\cdot C = \deg(K_{\cF_{C}}+\operatorname{Diff}(\cF,0)) = \deg\operatorname{Diff}(\cF,0)\geq 0,\]
    which is impossible. Here, the first inequality follows from the assumption, and $\operatorname{Diff}(\cF,0)\geq0$ is the effective divisor from adjunction Lemma~\ref{lem:adj_non-inv}. 
\end{proof}

\begin{remark} By foliated cone theorem, Proposition \ref{prop:foliatedMMP} holds in any dimension if $\cF$ has $\cF$-dlt singularities by the foliated cone theorem \cite{CS:fMMP3, CS:fMMPr1}. However, we have no control over the singularities of $\cF$ in general. 
\end{remark}

In the rest of the discussion, we fix $X\ra Y_+$ a $K_\cF$-MMP on a $(-K_X)$-slope unstable nonsingular weak del Pezzo surface $X$. We will derive a classification of $X\ra Z_+$ when $\dim Z_+=1$ in Section \ref{sec:classification}, and prove that $\dim Z_+=0$ never happens in Section \ref{sec:nonexist}.

\section{\texorpdfstring{$\dim Z_+=1$}{dim Z+=1}: The classification}\label{sec:classification}
Let $X\rightarrow Z_+$ be the fibration induced by a $K_\cF$-MMP as in Subsection \ref{subsec:mmp-wdP}. We assume in this section that $\dim Z_+=1$, so $Z_+=\bP^1$ as $X$ is rationally connected. Since $X$ is nonsingular, 
there is a diagram 
    \[\xymatrix{\varphi_+:X\ar[r]^\varphi & \widetilde{Y_+}\ar[r]^{\varphi'}\ar[d]_\chi&Y_+\ar[d]_{f_+}^{K_\cF\textnormal{-Mfs}}\\ & Y'=\bF_n\ar[r]^\pi& Z_+=\bP^1,}\]
    where $\varphi':\widetilde{Y_+}\to Y_+$ is the minimal resolution of $Y_+$ and $\chi:\widetilde{Y_+}\to Y'$ is obtained by running a relative minimal model program of $\widetilde{Y_+}$ over $\bP^1$. 
    Note that by Proposition \ref{prop:nPSEF} and \ref{prop:foliatedMMP}, $f_+$ has $\cF_+$-invariant rational fibers and $\chi\circ\varphi : X\rightarrow \widetilde{Y_+}\rightarrow Y'$ consists of a sequence of blowing down $(-1)$-curves sitting in fibers over $\bP^1$.

We now prove the classification in Theorem \ref{thm:nonsingular-wdP}.
\begin{theorem} With the setup in this section, $X$ is up to isomorphism either $\bF_n$ with $n\in\{0,1,2\}$ or obtained by performing at most three $2$-blowups from $\bF_n$, $n\in\{0,1,2\}$, with centers on distinct fibers and not on any $(-2)$-curve. In either case, $\cF$ is the pullback foliation of the relative tangent $T_{\bF_n/\bP^1}$. 
\end{theorem}
\begin{proof} We discuss what $X$ is in the following steps. 

{\bf Step 1.} Write $X=X_m\ra X_{m-1}\ra\dots\ra X_0=Y'$ with $m\geq0$, where each $X_{i+1}\rightarrow X_i$ is a nonsingular blowup at $p_i\in X_i$. Note that each $X_i$ is weak del Pezzo \cite[Proposition 8.1.2]{Dol} and in particular $Y'=\bF_n$ for some $n\in\{0,1,2\}$. The foliation $\cF$ pushes forward to a foliation $\cF'$ on $\bF_n$ induced by the fibration $\pi$ as the foliation $\cF_{\widetilde{Y_+}}$ is induced by the fibration $\pi\circ\chi$ and each prime exceptional divisor of $\chi$ is contained in a fiber of $\pi\circ\chi$. By \cite[Lemma 6.7]{AD:Ff}, we get $\cF' = T_{Y'/\bP^1}$ and $K_{\cF'}=K_{Y'/\bP^1}$. 

{\bf Step 2.} We may assume $X\neq Y'$. Note that any fiber of $X_i\rightarrow\bP^1$ has a simple normal crossing support since a general fiber is $\bP^1$. Let $N_i:=2K_{\cF_i}\cdot K_{X_i}-K_{X_i}^2$ for $0\leq i\leq m$. From Lemma~\ref{lem:blowup_N_difference}, $N_{i+1}=N_i\pm 1$ for $2\leq i \leq m$ since a singular fiber has at worst reduced singularities. By assumption, we have $N_0=2K_{\cF'}\cdot K_{Y'}-K_{Y'}^2=0$, $N_m\geq 0$, so in particular, $N_1=-1$ as all points on $X_0$ are nonsingular foliation points and $m\geq2$. 

{\bf Step 3.} Since $N_0=0>N_1=-1$ and $N_m\geq0$, there is $N_{i+1}>N_i$ for some $i$. This is only possible when $p_i$ is a reduced singularity, say $p_i\in F_1\cup F_2$. By \cite[Proposition 8.1.2]{Dol}, $p_i$ cannot lie on any $(-2)$-curves in $X_i$ since $X_{i+1}$ is weak del Pezzo. Hence, each $F_i$ can only be a $(-1)$-curve. In particular, this creates a singular fiber on $X$ that factors through a 2-blowup as defined in Definition \ref{def:2blowup}. 

{\bf Step 4.} Note that, however, any further blowup after a 2-blowup can only take place at a nonsingular point of the induced $(-1)$-curve as each $X_i$ is weak del Pezzo. This process can only decrease $N_i$, while the total effect of a 2-blowup contributes zero to the sequence of $N_i$'s. Hence, any singular fiber on $X$ must exactly come from a 2-blowup by $N_m\geq N_0=0$.

{\bf Step 5.} Since $-K_X$ is nef an big, we have $0<K_X^2=8-m$ and thus, $m\leq 7$. If $t$ is the number of nonsingular points among $p_i$'s, then $N_m=-t+(m-t)\geq0$ by Lemma \ref{lem:blowup_N_difference}. Hence, $t\leq 3$ and $X$ has at most three singular fibers. This completes the proof. 
\end{proof}

\begin{lemma}\label{lem:blowup_N_difference}
    Let $\cF$ be a foliation of rank one with only reduced singularities on a nonsingular projective surface $X$, $\pi: Y \to X$ be a blowup at $x$ with the exceptional divisor $E$, and $\cF_Y:=\pi^{-1}\cF$. 
    Define $N_X=2K_\cF\cdot K_X - K_X^2$ and $N_Y=2K_{\cF_Y}\cdot K_Y-K_Y^2$. 
    Then we have 
    \begin{align*}
        N_Y = \begin{cases}
            N_X-1 & \mbox{ if } x\notin\operatorname{Sing}(\cF),  \\
            N_X+1 & \mbox{ if } x\in\operatorname{Sing}(\cF).
        \end{cases}
    \end{align*}
\end{lemma}
\begin{proof}
    Note that if $x\notin\operatorname{Sinig}(\cF)$, then $K_{\cF_Y} = \pi^*K_{\cF}+E$ and $K_Y = \pi^*K_X+E$. 
    Thus, we have 
    \begin{align*}
        N_Y &= 2K_{\cF_Y}\cdot K_Y - K_Y^2 \\
        &= 2(\pi^*K_\cF+E)\cdot(\pi^*K_X+E) - (\pi^*K_X+E)^2 \\
        &= 2K_\cF\cdot K_X - K_X^2 + E^2 = N_X-1.
    \end{align*}
    If $x\in\operatorname{Sinig}(\cF)$, then $\cF$ has a reduced singularity at $x$ and $K_{\cF_Y} = \pi^*K_{\cF}$. A similar computation shows that $N_Y=N_X+1.$
\end{proof}

\section{\texorpdfstring{$\dim Z_+=0$}{dim Z+=0}: Non-existence}\label{sec:nonexist}
In this section, we prove that a Mori fiber space $Y_+\ra Z_+$ outcome of a $K_\cF$-MMP on a $(-K_X)$-slope unstable weak del Pezzo surface $X$ cannot have $\dim Z_+=0$. The proof depends on a detailed analysis of the singularities and geometry of the pair $(Y_+,\cF_+)$, based on the following key construction.

\begin{theorem}\cite[Section 3]{ACSS:PosM}\label{prop:*} Let $X$ be a normal variety and $\cF$ be an algebraically integrable foliation. There exists a {\bf Property $(*)$-modification} of $\cF$, i.e. a birational morphism $\pi: W\to X$ such that
    \begin{enumerate}[$(1)$]
        \item $W$ is $\bQ$-factorial and klt,
        \item $\cG:=\pi^{-1}\cF$ is induced by an equi-dimensional morphism $g: W\to Z$ over a nonsingular variety $Z$, 
        \item $(\cG,\sum\varepsilon(E)E)$, where the sum is over all $\pi$-exceptional prime divisors, is log canonical, 
        \item $K_{\cG}+\sum\varepsilon(E)E+G=\pi^*K_\cF$ for some effective and $\pi$-exceptional $G$.
    \end{enumerate} 
\end{theorem}

From the above discussions, we use the following diagram in the rest of our arguments:
\[\xymatrix{W \ar[rd]^-\pi \ar[dd]_-g & & & W_+ \ar[ld]_-{\pi_+} \ar[dd]^-{g_+} \\
    & X \ar[r]^-{\varphi_+} & Y_+ & \\
    \bP^1 & & & \bP^1}\]
where 
\begin{enumerate}
    \item $X$ is a nonsingular weak del Pezzo surface with a $(-K_X)$-slope unstable tangent bundle;
    
    \item $\cF\subseteq T_X$ is the maximal destabilizing sheaf and $X\ra Y_+$ is a $K_\cF$-MMP. By Proposition~\ref{prop:foliatedMMP}, $\varphi_+$ contracts only $\cF$-invariant curves;
   
    \item $\pi: W\to X$ (resp. $\pi_+: W_+ \to Y_+$) is a Property $(*)$ modification with $\cF_W=\pi^{-1}\cF$ (resp. $\cF_{W_+}=\pi_+^{-1}\cF_+$) induced by an equi-dimensional morphism $g: W\to \bP^1$ (resp. $g_+: W_+\to \bP^1$); 

   \item Let $F_W$ be a general fiber of $g$ and $F:=\pi_*F_{W}$ be a general leaf of $\cF$. Similarly, let $F_{W_+}$ be a general fiber of $g_+$ and $F_+:=(\pi_+)_*F_{W_+}$ be a general leaf of $\cF_+$.

    \item Finally, assuming $\dim Z_+=0$, $Y_+$ has Picard number one and $-K_{\cF_+}$ is ample.
\end{enumerate}

We will show that the $(Y_+,\cF_+)$ is toric and then rule out all possible toric geometry by a case-by-case study.

\begin{proposition}\label{prop:dim_Z+=0_imply_lc_fol} With the above setup, there is exactly one log canonical dicritical singularity on $Y_+$, which is a nonsingular point, and all other points of $Y_+$ are terminal foliation singularities. 
\end{proposition}
\begin{proof} The proof proceeds in the following steps: 

{\bf Step 1.} \emph{There is exactly one dicritical singularity $p_+$ on $Y_+$ over which there is exactly one non-foliation-invariant $\pi$-exceptional prime divisor $E_0^+$.} 

By the Property $(*)$ modification in Theorem \ref{prop:*}, we can write $(\pi_+)^*K_{\cF_+} = K_{\cF_{W_+}}+\Delta_+$ for some effective divisor $\Delta_+$, where $\operatorname{mult}_{E_i^+}\Delta^+\geq \varepsilon(E_i^+)$ for any $\pi_+$-exceptional prime divisor $E_i^+$.
    
As $\cF_+$ is Fano and algebraically integrable, by \cite[Proposition 5.3]{AD:Ff} $\cF_+$ is not induced by a morphism. Thus, as $\cF_{W_+}$ is non-dicritical, there exists a non-$\cF_{W_+}$-invariant $\pi_+$-exceptional prime divisor, say $E_0^+$. Take a general fiber $F_{W_+}$ of $g_+$, we have  
    \[0>K_{\cF_+}\cdot (\pi_+)_*F_{W_+} = (K_{\cF_{W_+}}+\Delta_+)\cdot F_{W_+} \geq -2 + \sum_i\varepsilon(E_i^+).\] 
Thus, there is exactly one non-invariant irreducible component of $\Delta_+$, that is $E_0^+$, and therefore, $p_+:=\pi_+(E_0^+)$ is the only dicritical singularity of  $\cF_+$ on $Y_+$.

{\bf Step 2.} \emph{The dicritical singularity $p_+\in Y_+$ is nonsingular.} 

Since $p_+$ is a dicritical point for $\cF_+$, there is a dicritical singularity $p$ for $\cF$ as $\varphi_+$ contracts only foliation-invariant divisors. Hence, there is a $\pi$-exceptional non-$\cF_W$-invariant divisor $E_0\subseteq W$ over $p$ and we can write $\pi^*K_{\cF} = K_{\cF_{W}}+a_0E_0+\Delta'$ for a $\pi$-exceptional divisor $\Delta'\geq0$ and some $a_0\geq 1$ such that $\Delta'\wedge E_0=0$. 

 As $X$ is nonsingular, we \emph{claim} that if a contracting curve $C$ in the process of the $K_\cF$-MMP $\varphi_+:X\ra Y_+$ contains a dicritical point, then the contracting locus of the corresponding $K_\cF$-negative extremal ray $R=\bR_{\geq 0}[C]$ covers the whole variety. The claim contradicts the birationality of $\varphi_+$, so $\varphi_+$ is an isomorphism near $p_+$ and $p_+\in Y_+$ is nonsingular. 

Let $C$ be a $K_\cF$-negative contracting curve, and $C_W$ be the proper transform of $C$ on $W$. Note that $C_W$ is $\cF_W$-invariant as $C$ is, and 
    \begin{align*}
        K_{\cF_W}\cdot C_W \leq K_\cF\cdot C = (K_{\cF_W}+a_0E_0+\Delta')\cdot C_W <0. 
    \end{align*}
We will show that \emph{$\cF_W$ is terminal at every closed point of $C_W$}. It then follows from \cite[Proposition 3.3]{CS:fMMPr1} that $C_W$ moves in a family, and hence the claim.

Assume now $\cF_W$ has a canonical non-terminal foliation singularity at $q\in C_W$. Note that $q\neq C_W\cap E_0$, otherwise $(\cF_W,\sum\varepsilon(E)E)$ is not log canonical, contradicting to that $\pi$ is a Property $(*)$-modification. Let $r$ be the index of $K_W$ at $E_0\cap C_W$. Then we have 
    \begin{align*}
        0 &> K_\cF\cdot C \\
        &=(K_{\cF_W}+a_0E_0+\Delta')\cdot C_W \\
        &\geq K_{\cF_W}\cdot C_W + E_0\cdot C_W \\
        &= \deg(K_{C_W^\nu}+{\rm Diff}_{C_W}(\cF_W)) + E_0\cdot C_W \\
        &\geq \big(-2+\frac{r-1}{r}+1\big)+\frac{1}{r}=0, 
    \end{align*}
where the second equality follows from the precise adjunction Lemma~\ref{lem:adjunction}, the term $\frac{r-1}{r}$ is from the terminal foliation point $C_W\cap E_0$, and the term $1$ is from the canonical non-terminal point $q$. This is absurd. 

{\bf Step 3.} \emph{$\cF_+$ is terminal on $Y_+\setminus\{p_+\}$.}

We write $\pi_+^*K_{\cF_+} = K_{\cF_{W_+}}+a_0^+E_0^++\Delta'_+$ where $a_0^+\geq 1$ and $\Delta'_+\wedge E_0^+=0$. Consider $F_+$ a leaf of $\cF_+$ and $q\in F_+$ a point other than $p_+$. If $q$ is not a terminal foliation singularity and $r$ is the index of $K_{W_+}$ at $E_0^+\cap F_{W_+}$, then by precise adjunction Lemma~\ref{lem:adjunction}, we have 
    \begin{align*}
        0 &> K_{\cF_+}\cdot F_+ \\ 
        &= \pi_+^*K_{\cF_+}\cdot F_{W_+} \\
        &= (K_{\cF_{W_+}}+a_0^+E_0^++\Delta'_+)\cdot F_{W_+} \\
        &= \deg(K_{F_{W_+}^\nu} + {\rm Diff}_{F_{W_+}}(\cF_{W_+})) + a_0^+E_0^+\cdot F_{W_+} + \Delta'_+\cdot F_{W_+} \\
        &\geq \big(-2+\frac{r-1}{r}+1\big)+\frac{a_0^+}{r}+0\geq 0,
    \end{align*}
where ${\rm Diff}_{F_{W_+}}(\cF_{W_+})$ is effective, the term $\frac{r-1}{r}$ is from $E_0^+\cap F_{W_+}$, and the term $1$ is from $q$. This is absurd.

{\bf Step 4.} \emph{$p_+$ is an lc foliation singularity}

Recall $\pi_+^*K_{\cF_+} = K_{\cF_{W_+}}+\Delta_+$ with $\operatorname{mult}_{E_i^+}\Delta^+\geq \varepsilon(E_i^+)$ for any $\pi_+$-exceptional divisor $E_i^+$. Note that all points on $F_+$ are nonsingular points of $Y_+$ and thus $K_{\cF_+}\cdot F_+\in\mathbb{Z}.$ Moreover, 
    \begin{align*}
        0 > K_{\cF_+}\cdot F_+ = \pi_+^*K_{\cF_+}\cdot F_{W_+} = (K_{\cF_{W_+}}+\Delta_+)\cdot F_{W_+} \geq -2+\sum\varepsilon(E_i^+)
    \end{align*}
and thus, $E_0^+$ from {\bf Step 1} is the unique $E_i^+$ with $\varepsilon(E_0^+)=1$ and it follows that $\operatorname{mult}_{E_0^+}\Delta_+=1$. 
 
Let $G$ be the dual graph of $\pi_+$-exceptional curves over $p_+$ and $d_i$ be the distance in $G$ between the vertices $[E_i^+]$ and $[E_0^+]$. Write $\Delta_+=E_0^++\sum_ia_iE_i^+$ and let $r_{ij}$ be the index of $K_{W_+}$ at $E_i^+\cap E_j^+$. Note that $G$ is connected and by foliation adjunction Lemma \ref{lem:adj_non-inv} on the non-invariant divisor $E_0^+$, 
    \begin{align*}
        0 &= \pi_+^*K_{\cF_+}\cdot E_0^+ \\
        &= (K_{\cF_{W_+}}+\sum a_iE_i^+)\cdot E_0^+ \\
        &= (K_{\cF_{W_+}}+E_0^+)\cdot E_0^++\sum_{i: d_i=1}a_iE_i^+\cdot E_0^+ \\
        &\geq\sum_{i: d_i=1}\frac{a_i}{r_{i0}}\geq 0.
    \end{align*}
So we have $a_i=0$ for those $i$ with $d_i=1$. 

Now suppose $a_i=0$ for those $i$ with $1\leq d_i\leq d-1$ for some integer $d\geq2$. For any vertex $[E^+_i]$ with $d_i=d$, $E_i^+$ is invariant and we have a vertex $[E^+_j]$ with $d_j=d-1$ connecting to $[E^+_i]$. Thus, by the induction hypothesis, we have 
    \[0=\pi_+^*K_{\cF_+}\cdot E^+_j = (K_{\cF_{W_+}}+\sum a_iE_i^+)\cdot E^+_j\geq 0+\frac{a_i}{r_{ij}}\geq 0,\]
where $K_{\cF_{W_+}}\cdot E^+_j\geq 0$ follows from the precise adjunction Lemma~\ref{lem:adjunction} and the fact that there are two canonical non-terminal singularities on $E_j^+$ as the intersection points of other components in $G$. 

Hence, $a_i=0$ for all $i$ with $d_i\geq 1$, near $p_+$ we see $\pi_+^*K_{\cF_+}=K_{\cF_{W_+}}+E_0^+$ is log canonical, and we conclude that $p_+$ is a log canonical foliation singularity.
\end{proof}

Our next goal is to demonstrate that $(Y_+,\cF_+)$, as in Proposition \ref{prop:dim_Z+=0_imply_lc_fol}, is indeed toric. We will prove that $W_+$ is toric and deduce that $Y_+$ is also toric. The following lemma, when applied to $Y_+$, enables us to control singularities on $W_+$ and hence on $Y_+$. Intuitively, as $\rho(Y_+)=1$, a necessary condition for $Y_+$ to be toric is that there are at most three singular points coming from invariant points. Cf. Proposition \ref{prop:dim_Z+=0_imply_toric}.

\begin{lemma}\label{lem:lc-chain}
Let $(X,\cF,p)$ be a germ of foliation on a nonsingular surface $X$ and $p\in X$ be a strict log canonical foliation singularity. 
Then, the dual graph of the exceptional divisors on the minimal resolution of $\cF$ over $p$ is a chain. 
\end{lemma}
\begin{proof}
Let $\pi : Y \to X$ be the minimal resolution of $\cF$ and $\bigcup_iE_i$ be the exceptional divisors over $p$. As $\cF$ is strict log canonical at $p$, by Theorem~\ref{lc_sing}, there is a unique non-invariant divisor $E_0$, and the dual graph of $\bigcup_iE_i$ is star-shaped. 

As both $Y$ and $X$ are nonsingular surfaces, $\pi=\pi_\ell\circ\cdots\circ\pi_0$ is a composition of $\ell$ nonsingular blowups $\pi_j:X_{j+1}\to X_j$ with $X_0:=X$ and $X_{\ell+1}:=Y$. Thus, there is at least one $E_i$ being the exceptional divisor of $\pi_\ell$ so that $E_i^2=-1$. If $E_i$ is invariant, then $X_\ell$ comes from the contraction of a $(-1)$-$\cF$-curve has at worst reduced singularities, and this violates the minimality of $\pi : Y \to X$. Hence, it is only possible that $E_0$ is the exceptional divisor of $\pi_\ell$, $E_0^2=-1$, and all other $E_i$ have $E_i^2 \leq -2$. 

Suppose for the sake of contradiction that there are at least three branches of the graph of $\bigcup_iE_i$. Then the center of $E_0$ on $X_\ell$ is a point which is contained in at least three prime exceptional divisors of $\pi_{\ell-1}\circ\cdots\circ\pi_0$. This is impossible since the dual graph of the exceptional divisors of $X_\ell\ra X$ is a tree. 
\end{proof}

\begin{proposition}\label{prop:dim_Z+=0_imply_toric} With the set up in Proposition \ref{prop:dim_Z+=0_imply_lc_fol}, $(Y_+,\cF_+)$ is a toric foliation and $|\Sing(Y_+)|\leq2.$
\end{proposition}
\begin{proof}
By Proposition~\ref{prop:dim_Z+=0_imply_lc_fol}, there is exactly one dicritical point $p_+$ for $\cF_+$, $p_+\in Y_+$ is nonsingular, and $\cF_+$ is strictly log canonical. By Proposition \ref{prop:dim_Z+=0_imply_lc_fol}, the minimal resolution $\widetilde{\pi}: \widetilde{W} \to Y_+$ of $\cF_+$ at $p_+$ gives a Property $(*)$ modification and the associated dual graph is a chain by Lemma \ref{lem:lc-chain}. By contracting down all the invariant $\widetilde{\pi}$-exceptional divisors, the induced morphism $\pi_+: W_+\to Y_+$ is still a Property $(*)$ modification so we can assume $\rho(W_+)=2$. Note that then $\pi_+^*K_{\cF_+} = K_{\cF_{W_+}}+E_0^+$, $\Sing(W_+)\cap E_0^+=\{p_1,\dots, p_s\}$ has $s\leq2$, $\cF_{W_+}$ is terminal at all $p_i$'s, and by Theorem \ref{can_sing} each $p_i\in W_+$ is a cyclic quotient singularity. 

Consider now a log pair $(W_+,F_1+F_2+E_0^+)$ with $F_1\neq F_2$ and $\{p_1,\dots,p_s\}\subseteq E_0^+\cap(F_1\cup F_2).$ 
Let $G=\bigcup_{i=0}^mG_i$ be the prime $\widetilde{\pi}$-exceptional divisors. 
Note that for $\widetilde{F_j}$ the strict transform of $F_j$, $j=1,2$, the pair $(\widetilde{W},\widetilde{F_1}+\widetilde{F_2}+\sum_{i=0}^mG_i)$ is log canonical and 
\[(K_{\widetilde{W}}+\widetilde{F_1}+\widetilde{F_2}+\sum_{i=0}^mG_i)\cdot G_j = 0,\ j=0,\dots,m.\]
Hence $K_{\widetilde{W}}+\widetilde{F_1}+\widetilde{F_2}+\sum_{i=0}^mG_i = \widetilde{\pi}^*(K_{Y_+}+L_1+L_2)$ with $L_i=(\pi_+)_*F_i$. 
Therefore, $(Y_+,L_1+L_2)$ is log canonical. 

Since $(K_{Y_+}+L_1+L_2)\cdot \pi_*F= (K_{W_+}+E_0^++F_1+F_2)\cdot F=-1$ for a general fiber $F$ of $g_+$ and $\rho(Y_+)=1$, the divisor $K_{Y_+}+L_1+L_2$ is anti-ample.
Hence, there is an effective divisor $C\sim_\bQ -(K_{Y_+}+L_1+L_2)$ such that $K_{Y_+}+L_1+L_2+\varepsilon C$ is anti-ample and log canonical for a sufficiently small $\varepsilon>0$. 
By \cite{BMSZ:Toric}, as the complexity $c(Y_+,L_1+L_2+\varepsilon C)= 1-\varepsilon<1$, there is a toric log Calabi-Yau pair $(Y_+,\Delta)$ with $\Delta\geq L_1+L_2$ such that $\Delta-L_1-L_2$ is supported on $C$. In particular, both $Y_+$ and $W_+$ are toric varieties, $g_+$ is a toric morphism, and $\cF_+$ is a toric foliation. 
    
Note that all singularities of toric surfaces are torus-invariant points and $|\Sing(Y_+)|\leq 3$ by $\rho(Y_+)=1$. As $p_+$ is a nonsingular torus-invariant point, we see $|\Sing(Y_+)|\leq 2.$
\end{proof}

\begin{theorem}\label{prop:dim_Z+=1}
Let $X$ be a nonsingular $(-K_X)$-slope unstable weak del Pezzo surface and $\cF\subseteq T_X$ be the maximal destabilizing subsheaf. Suppose $\varphi_+:X\to Y_+$ is a $K_\cF$-MMP with the Mori fiber space outcome $Y_+\ra Z_+$. Then $\dim Z_+=1$.     
\end{theorem}
\begin{proof}
Assume by contradiction that $\dim Z_+=0$. By Proposition~\ref{prop:dim_Z+=0_imply_toric}, $Y_+$ is a toric variety of $\rho(Y_+)=1$ with at most two singular points and $\cF_+:=(\varphi_+)_*\cF$ is a toric foliation. Note that $Y_+$ has canonical singularities by Lemma~\ref{D_MMP_canonical} and hence is Gorenstein. By the classification of Gorenstein del Pezzo toric surfaces (cf. \cite[Theorem 8.3.7]{CLS}), there are exactly 16 isomorphism classes:
    \begin{center}
    \includegraphics[scale=0.3]{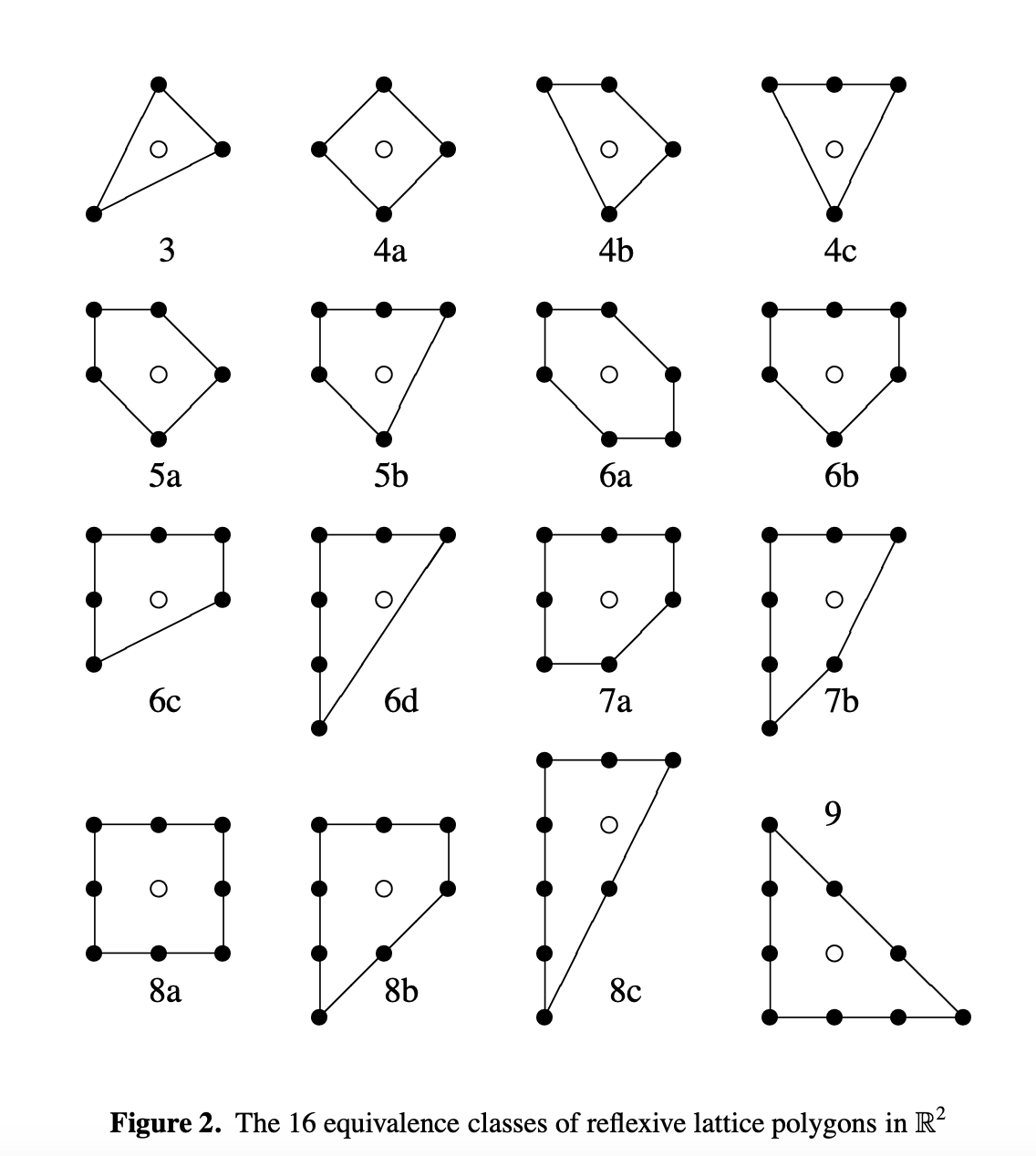}
    \end{center}
Since $Y_+$ has Picard number one and $|\Sing(Y_+)|\leq 2$ by Proposition \ref{prop:dim_Z+=0_imply_toric}, we only need to consider no. 6d, 8c, and 9. 

Let $\Sigma$ be the fan of $Y_+$ and $V\subseteq N_\bC$ be the subspace corresponding to the toric foliation $\cF_+$ as in Proposition \ref{prop:toric-ss}.
By Proposition~\ref{prop:toric_fol_can_div}, $V=\bC v_i$ for some $i=1,2,3$ as $-K_{\cF_+}$ is ample, where $v_i$'s are the primitive generators of the rays in $\Sigma(1)$. 
By Proposition~\ref{prop:tor_fol_sing}, there is only one foliation singularity, being a dicritical point at a nonsingular torus invariant point $p_+\in Y_+$, of $\cF_+$. We collect below all possibilities of $(\Sigma, V)$: 

\[
\begin{tabular}{|c|c|c|c|c|c|}
\hline
 case & $v_1$ & $v_2$ & $v_3$  & $V$ & $p_+$\\ \hline
 6d& $(1,0)$ & $(0,-1)$ & $(-3,2)$ & $\bC v_3$ & $\sigma_{12}$  \\ \hline
 8c& $(1,0)$ & $(0,-1)$ & $(2,-1)$  &$\bC v_i, i=2,3$ & 
 $\sigma_{\widehat{i}}=\sum_{j\neq i}\bR_{\geq0} v_j$ 
 \\ \hline
 9& $(1,0)$ & $(0,-1)$ & $(-1,1)$  & $\bC v_i,i=1,2,3$ &
 $\sigma_{\widehat{i}}=\sum_{j\neq i}\bR_{\geq0} v_j$
 \\ \hline
\end{tabular}\]

We factor $X\to Y_+$ through $\widetilde{Y_+}$, the minimal resolution of $Y_+$. 
Write $X\to \widetilde{Y_+}$ as $X:=X_\ell\xrightarrow{\pi_{\ell-1}} X_{\ell-1}\xrightarrow{\pi_{\ell-2}}\cdots\xrightarrow{\pi_0}X_0:=\widetilde{Y_+}$, where each $\pi_i$ is a nonsingular blowup at $q_i$ with $E_i$ the unique $\pi_i$-exceptional divisor for $0\leq i\leq \ell-1$. Note that $q_i$ doesn't lie on any $(-2)$-curves as each $-K_{X_i}$ is nef \cite[Proposition 8.1.2]{Dol}.

Denote $\cF_i$ the induced foliation on $X_i$ and $N_i:=2K_{\cF_i}\cdot K_{X_i}-K_{X_i}^2$, $0\leq i\leq\ell$. As $N_\ell\geq0$ and $N_0=-4,-4,-3$ respectively for no. 6d, 8c, and 9, we must have $\ell\geq1$. Note that $\cF_{\widetilde{Y_+}}$ remains a toric foliation and so $\Sing(\cF_{\widetilde{Y_+}})$ consists of torus invariant points. 

Let $\alpha$ be the number of nonsingular foliation points among $\{q_0,\dots,q_{\ell-1}\}$. We recall that no curve contracted by the $K_\cF$-MMP is over the dicritical point $p_+$ (cf. the proof of {\bf Step 2} in Proposition~\ref{prop:dim_Z+=0_imply_lc_fol}). 
Since $\cF_{\widetilde{Y_+}}$ is toric and all foliation singularities of $\cF_{\widetilde{Y_+}}$ are either the dicritical point or points on some $(-2)$-curves extracted by the minimal resolution $\widetilde{Y_+}\to Y_+$, no $q_i$ is over $p_+$ and hence, $q_0$ must be a nonsingular foliation point and $\alpha\geq 1$. 

On the other hand, $0\leq N_\ell = N_0 - \alpha + (\ell-\alpha)$ by Lemma \ref{lem:blowup_N_difference}, $1\leq \ell=K_{Y_+}^2-K_X^2<K_{Y_+}^2$, and thus $0\leq\alpha<\frac{1}{2}(N_0+K^2_{Y_+})$. This immediately implies that no 6d is impossible as $1\leq\alpha<\frac{1}{2}(-4+K^2_{Y_+})=1.$
The data of remaining cases 8c and 9 is summarized below. 
\[
\begin{tabular}{|c|c|c|c|c|}
\hline
 case no. & $N_0$ & $K_{Y_+}^2$ & $\alpha$ & $2\alpha-N_0\leq\ell< K_{Y_+}^2$ \\ \hline
 8c& $-4$ & 8 & 1 & $ 6\leq\ell<8$ \\ \hline
 9& $-3$ & 9 & 1, 2 & $ 5, 7\leq\ell<9$ \\ \hline
\end{tabular}\]

We will show that none of the above cases can happen under the following constraints:
\begin{enumerate}
    \item $N_\ell= 2K_\cF\cdot K_X-K_X^2\geq 0$;
    \item $q_i$ is not over $p_+$ and $q_0$ is a nonsingular foliation point; 
    \item $q_i$ doesn't lie on any $(-2)$-curves as each $-K_{X_i}$ is nef;
    \item $N_{i+1}=N_i+ 1$ (resp. $N_i-1$) if $q_i\in\Sing(\cF_i)$ (resp. $q_i\notin\Sing(\cF_i)$).
\end{enumerate}

{\bf Case (no. 8c):} Let $\Sigma=\Sigma_{-1}$ be the fan of $Y_+\cong\overline{\bF}_2$, $\cF_+$ be associated to $V_1=\bC v_3$ or $V_2=\bC v_2$, and recall $\alpha=1$. 
(See Figure~\ref{fig:Sigma2} for the case $V_1=\bC v_3$.) 
It suffices to consider the case $V_1$ since $(Y_+=X_{\Sigma_{-1}},\cF_{V_1})$ is isomorphic to $(X_{\Sigma_{-1}},\cF_{V_2})$ as toric foliations. 
The $1$-dimensional cones in the fan $\Sigma_0$ of the minimal resolution $\widetilde{Y_+}$ are the elements in $\Sigma_{-1}(1)\cup\{\rho_4\}$. As $q_0$ is a nonsingular foliation point, there is a unique leaf $L_0\subseteq\widetilde{Y_+}$ containing $q_0$. 
If $L_0$ is torus invariant, then $L_0=D_{\rho_i}$ for $i=1,2,3,4$. 
Note that $L_0\neq D_{\rho_4}$ since $q_0\in L_0$ and $D_{\rho_4}^2=-2$. 
Also $L_0\neq D_{\rho_3}$ as $D_{\rho_3}$ is not foliation invariant by \cite[Corollary 3.3]{CC:Tf}. 
By a direct computation, $D_{\rho_1}^2=2$ and $D_{\rho_2}^2=0$. 

\begin{claim}
    $L_0^2=2$ if $L_0$ is not torus invariant. 
\end{claim}
\begin{proof}
The fan $\Sigma_2$ is obtained from $\Sigma_0$ by first performing the star subdivision along $\rho_5$ and then $\rho_6$, cf. Figure~\ref{fig:Sigma2}. 
Equivalently, $X_{\Sigma_2} \to X_{\Sigma_0}=\widetilde{Y_+}$ is a composition of two nonsingular blowups at torus invariant points. 
The lattice $N_2 :=\bZ v_1\oplus\bZ v_3$ admits a natural quotient map $g: N_2 \to N_1=\bZ v_1$, which induces a toric morphism $f: X_{\Sigma_2}\to \bP^1$. By \cite[Proposition 3.13(1)]{CC:Tf}, the toric foliation $\cF_{V_1}$ on $X_{\Sigma_2}$ is induced by the fibration $f$. Hence, the strict transformation $L_0''$ on $X_{\Sigma_2}$ is a general fiber of $f$. Therefore, $0=(L_0'')^2=L_0^2-2$, and the claim follows. 
\end{proof}

\begin{figure}
\centering
    \begin{tikzpicture}
        \def\r{0.75}
        \draw[thick,green,->] (0,0) -- (-4*\r,0) node[anchor=north] {$\rho_4$};
        \filldraw[black] (-1*\r,0) circle (2pt);
        \draw[blue] (-4*\r,2*\r) -- (4*\r,-2*\r) node[anchor=west] {$V_1$};
        
        \draw[thick,->] (0,0) -- (3*\r,0) node[anchor=west] {$\rho_1$};
        \filldraw[black] (\r,0) circle (2pt) node[anchor=south] {$(1,0)$};
        \draw[thick,->] (0,0) -- (0,-2*\r) node[anchor=north] {$\rho_2$};
        \filldraw[black] (0,-\r) circle (2pt) node[anchor=east] {$(0,-1)$};
        \draw[thick,->] (0,0) -- (-4*\r,2*\r) node[anchor=south] {$\rho_3$};
        \filldraw[black] (-2*\r,\r) circle (2pt) node[anchor=south] {$(-2,1)$};
        \draw[thick,red,->] (0,0) -- (2*\r,-2*\r) node[anchor=north] {$\rho_5$};
        \filldraw[black] (\r,-1*\r) circle (2pt) node[anchor=north] {$(1,-1)$};
        \draw[thick,orange,->] (0,0) -- (3*\r,-1.5*\r) node[anchor=west] {$\rho_6$};
        \filldraw[black] (2*\r,-1*\r) circle (2pt) node[anchor=south] {$(2,-1)$};
    \end{tikzpicture}
    \caption{$\Sigma_0(1)=\{\rho_1,\rho_2,\rho_3,\rho_4\}$; $\Sigma_2(1)=\{\rho_i:1\leq i\leq 6\}$.}\label{fig:Sigma2}
\end{figure}
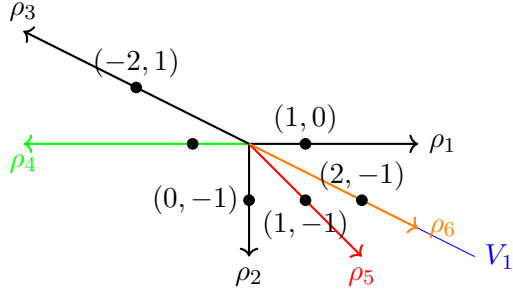

In either case, whether $L_0$ is torus invariant or not, we have $L_0^2\in\{0,2\}$. 
Let $L_i$ be the strict transform of $L_0$ on $X_i$. 
As $\alpha=1$ and $q_0$ is a nonsingular foliation point, all $q_i$'s are foliation singularities for $i\geq 1$. 
Since $q_1$ does not lie on any $(-2)$-curves, $q_1=L_1\cap E_1$, and similarly, $q_2=L_2\cap E_2$ by the given constraints. 
Now on $X_3$, if $L_0^2=0$, then $L_2^2=-2$ and all foliation non-dicritical singularity lie on some $(-2)$-curve. 
Hence, we cannot proceed any further with the blowup, contradicting $\ell\geq 6$. 
Therefore, $L_0^2=2$ and by the same consideration, the procedure above can only proceed till $\ell=4$, contradicting $\ell\geq 6$. (See Figure~\ref{fig:no_8c_L^2=2}.)

\begin{figure}
    \centering
    \begin{tikzpicture}
        \def\r{0.9}
        \filldraw[black] (0,0) circle (2pt);
        \draw (0.5*\r,-0.5*\r) -- (-1.5*\r,1.5*\r) node[anchor=south] {$L_1$};
        \draw (-0.5*\r,-0.5*\r) -- (1.5*\r,1.5*\r) node[anchor=south] {$E_1$};
        \draw (0,0) node[anchor=east] {$q_1$};

        \draw[orange] (-0.5*\r,0.5*\r) circle (9pt);
        \draw[orange] (-0.5*\r,0.5*\r) node {$1$};
        \draw[orange] (0.5*\r,0.5*\r) circle (9pt);
        \draw[orange] (0.5*\r,0.5*\r) node {$-1$}; 

        \draw[<-] (1.5*\r,0.5*\r) -- (2.5*\r,0.5*\r);

        \filldraw[black] (4*\r,0) circle (2pt);
        \draw (4.5*\r,-0.5*\r) -- (2.5*\r,1.5*\r) node[anchor=south] {$L_2$};
        \draw (3.5*\r,-0.5*\r) -- (5.5*\r,1.5*\r) node[anchor=south] {$E_2$};
        \draw (4.5*\r,1.5*\r) -- (6.5*\r,-0.5*\r) node[anchor=north] {$(\pi_1)_*^{-1}E_1$};
        \draw (4*\r,0) node[anchor=east] {$q_2$};

        \draw[orange] (3.5*\r,0.5*\r) circle (9pt);
        \draw[orange] (3.5*\r,0.5*\r) node {$0$};
        \draw[orange] (4.5*\r,0.5*\r) circle (9pt);
        \draw[orange] (4.5*\r,0.5*\r) node {$-1$}; 
        \draw[orange] (5.5*\r,0.5*\r) circle (9pt);
        \draw[orange] (5.5*\r,0.5*\r) node {$-2$}; 

        \draw[<-] (6.5*\r,0.5*\r) -- (7.5*\r,0.5*\r);

        \filldraw[black] (9*\r,0) circle (2pt);
        \draw (9.5*\r,-0.5*\r) -- (7.5*\r,1.5*\r) node[anchor=south] {$L_3$};
        \draw (8.5*\r,-0.5*\r) -- (10.5*\r,1.5*\r) node[anchor=south] {$E_3$};
        \draw (9.5*\r,1.5*\r) -- (11.5*\r,-0.5*\r); 
        \draw (10.5*\r,-0.5*\r) -- (12.5*\r,1.5*\r); 
        \draw (9*\r,0) node[anchor=east] {$q_3$};

        \draw[orange] (8.5*\r,0.5*\r) circle (9pt);
        \draw[orange] (8.5*\r,0.5*\r) node {$-1$};
        \draw[orange] (9.5*\r,0.5*\r) circle (9pt);
        \draw[orange] (9.5*\r,0.5*\r) node {$-1$}; 
        \draw[orange] (10.5*\r,0.5*\r) circle (9pt);
        \draw[orange] (10.5*\r,0.5*\r) node {$-2$}; 
        \draw[orange] (11.5*\r,0.5*\r) circle (9pt);
        \draw[orange] (11.5*\r,0.5*\r) node {$-2$}; 
    \end{tikzpicture}
    \caption{The visualization of {\bf Case (no. 8c)} when $L_0^2=2$.}\label{fig:no_8c_L^2=2}
\end{figure}
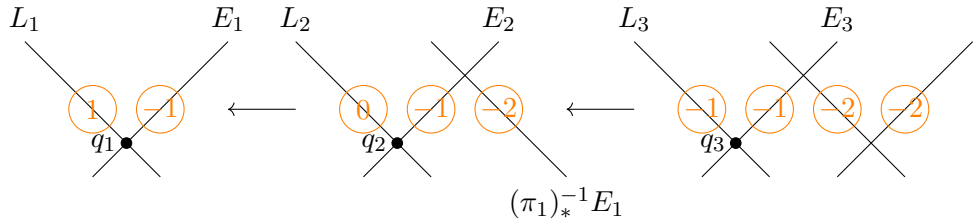

{\bf Case (no. 9):} $Y_+=\bP^2$, $\cF_+$ is the radial foliation, and $\alpha=1$, $2$. 
Let $L_0$ be the leaf containing $q_0$. 
Then $L_0^2=1$. 
Let $L_i$ be the strict transform of $L_0$ on $X_i$. 
When $\alpha=1$, then by the same argument as in the {\bf Case (no. 8c)}, we can not proceed with the blowup when $\ell=3$, which contradicts $\ell\geq 5$. 
When $\alpha=2$, then exactly one of $q_i\neq q_0$, say $q_{i_0}$, is a nonsingular foliation point. 
We may assume $q_{i_0}=q_1$, so for $i\geq 2$, the points $q_i$'s are not nonsingular foliation points. 
If $q_1$ is over $q_0$, then $L_2^2=L_0^2-2=-1$. 
As $q_2$ is not a nonsingular foliation point, $q_2\in\{L_2\cap(\pi_1)_*^{-1}E_1, L_2\cap E_2\}$. 
In either case, any foliation singularity on $X_3$ lies on some $(-2)$-curve, and thus, we can only proceed with the blowup till $\ell=3$, contradicting $\ell\geq 5$. 
So $q_1$ is not over $q_0$. 
Let $L'_1$ be the leaf of $\cF_1$ containing $q_1$. 
Note that $(L'_1)^2=1$. 
Then, the same argument again shows that we can only proceed with the blowup up to $\ell=3+3=6$, contradicting $\ell\geq 7$. 
\end{proof}

\bibliographystyle{alpha}
\bibliography{Unstable_Fano}
\end{document}